\documentclass{amsart}
\usepackage{amssymb}
\usepackage{latexsym}
\usepackage{amsmath,amscd,amsthm}
\usepackage{amsfonts,epsfig,graphicx,amssymb,pict2e}
\date{\today}

\newcommand{\Z}{{\mathbb Z}}
\newcommand{\R}{{\mathbb R}}

\newcommand{\T}{{\mathbb T}}

\newtheorem{theorem}{Theorem}

\newtheorem{lemma}{Lemma}[section]
\newtheorem{prop}[lemma]{Proposition}

\newtheorem{Lm}{Lemma}[section]
\newtheorem{Thm}[Lm]{Theorem}
\newtheorem{Prop}[Lm]{Proposition}

\newtheorem{Rem}[Lm]{Remark}
\newtheorem{Cor}{Corollary}

\newtheorem{Def}{Definition}
\def\bdef{\begin{Def}}
\def\endef{\end{Def}}
\def\bthm{\begin{Thm}}
\def\ethm{\end{Thm}}
\def\bprop{\begin{Prop}}
\def\enprop{\end{Prop}}
\def\blm{\begin{Lm}}
\def\elm{\end{Lm}}
\def\bcor{\begin{Cor}}
\def\ecor{\end{Cor}}
\def\brm{\begin{Rem}}
\def\erm{\end{Rem}}
\def\bfig{\begin{picture}}
\def\efig{\end{picture}}
\def\beq{\begin{eqnarray}}
\def\eneq{\end{eqnarray}}
\def\beal{\begin{aligned}}
\def\enal{\end{aligned}}

\begin{document}
\title[The Integrated Density of States of the Fibonacci Hamiltonian]{H\"older Continuity of the
Integrated Density of States for the Fibonacci Hamiltonian}

\author{David Damanik}

\address{Department of Mathematics, Rice University, Houston, TX~77005, USA}

\email{damanik@rice.edu}

\thanks{D.\ D.\ was supported in part by NSF grants DMS--0800100 and DMS--1067988.}

\author{Anton Gorodetski}

\address{Department of Mathematics, University of California, Irvine CA 92697, USA}

\email{asgor@math.uci.edu}

\thanks{A.\ G.\ was supported in part by NSF grants DMS--0901627 and IIS-1018433.}

\date{\today}

\begin{abstract}
We prove H\"older continuity of the integrated density of states for the Fibonacci Hamiltonian for any positive coupling, and obtain the asymptotics of the H\"older exponents for large and small couplings.
% The key
%ingredients in the proof are scaling properties of the spectrum
%and lower bounds for the Lebesgue measure of individual bands of
%the spectra of the periodic approximants.
\end{abstract}

\maketitle

\section{Introduction}

In this paper we are interested in regularity properties of the integrated density of states associated with the Fibonacci Hamiltonian. We will show that it is uniformly H\"older continuous and provide explicit estimates for the H\"older exponent in terms of the coupling constant.

To motivate our study, let us consider an invertible ergodic transformation $T$ of a probability measure space $(\Omega,d\mu)$ and a bounded measurable function $f : \Omega \to \R$. One associates a family of discrete Schr\"odinger operators on the line as follows: For $\omega \in \Omega$, the potential $V_\omega : \Z \to \R$ is given by $V_{\omega}(n) = f(T^n \omega)$ and the operator
$H_\omega$ in $\ell^2(\Z)$ acts as
$$
[H_\omega \phi](n) = \phi(n+1) + \phi(n-1) + V_\omega(n) \phi(n).
$$
An important quantity associated with such a family of operators, $\{H_\omega\}_{\omega \in \Omega}$, is given by the integrated density of states, which is defined as follows; compare \cite{as,cfks}. Define the measure $dN$ by
\begin{equation}\label{e.idsspectraldef}
\int g(\lambda) dN(\lambda) = \int \langle \delta_0 , g(H_\omega) \delta_0 \rangle \, d\mu(\omega).
\end{equation}
The integrated density of states (IDS), $N$, is then given by
\begin{equation}\label{e.idsspectraldef2}
N(E) = \int \chi_{(- \infty, E]} (\lambda) \, dN(\lambda).
\end{equation}
The terminology is explained by
\begin{equation}\label{idsform}
N(E) = \lim_{n \to \infty} \frac{\# \{ \text{eigenvalues of } H_{\omega,[1,n]} \le E \}}{n} \; \; \text{ for $\mu$-a.e.\ } \omega \in \Omega,
\end{equation}
where $H_{\omega,[1,n]}$ denotes the restriction of $H_\omega$ to the interval $[1,n]$ with Dirichlet boundary conditions. It is a basic result that the IDS is always continuous \cite{as}; see \cite{ds} for a very short proof that also works in higher dimensions. Craig and Simon, \cite{cs} (see \cite{cs2} for the multi-dimensional case), have shown that more is true. Using the Thouless formula, they proved that $N$ is $\log$-H\"older continuous, that is, there is a constant $C$ such that
\begin{equation}\label{logholder}
| N(E_1) - N(E_2) | \le C \left( \log |E_1 - E_2|^{-1} \right)^{-1}
\end{equation}
for all $E_1,E_2$ with $|E_1 - E_2| \le 1/2$.

In this general setting, the bound \eqref{logholder} is essentially optimal; see Craig \cite{c} and Gan and Kr\"uger \cite{GK11}. However, for concrete models, one may hope to improve upon \eqref{logholder}. Roughly speaking, one expects stronger regularity properties of the IDS the more random the stochastic process $V_\omega(n)$ is. In the i.i.d.\ situation, it is always H\"older continuous, that is,
\begin{equation}\label{holder}
| N(E_1) - N(E_2) | \le C |E_1 - E_2|^{\gamma}
\end{equation}
for some $C < \infty$ and $\gamma > 0$, as shown by Le Page \cite{l}. If the single-site distribution is nice enough, Simon-Taylor \cite{st} and Campanino-Klein \cite{ck} proved that $N$ is $C^\infty$. We refer the reader to the survey article \cite{s1} by Simon, which describes the regularity results for the IDS that had been obtained by the mid-1980's.

More recently, there has been renewed interest in the problem of proving regularity better than \eqref{logholder} for some classes of ergodic Schr\"odinger operators with little randomness. This was initiated by Goldstein and Schlag, who proved H\"older continuity of the IDS, \eqref{holder}, for analytic quasi-periodic potentials in the regime of positive Lyapunov exponents \cite{gs}. This paper was followed by \cite{b,b2,bgs, GS08} who improved the estimate in some cases or proved regularity results for different models (e.g., with $T$ given by a skew-shift on $\T^2$). Also see Hadj-Amor \cite{H09} and Avila-Jitomirskaya \cite{AJ10} for results for analytic quasi-periodic potentials in the regime of zero Lyapunov exponents, and Schlag \cite{s4} for results for analytic quasi-periodic models at large coupling in two dimensions.

Due to the typical presence of a dense set of gaps in the spectrum, one does not hope for more than H\"older regularity for quasi-periodic (or, more generally, almost-periodic) models.

\bigskip

Our objective here is to study the regularity properties of the IDS for a prominent quasi-periodic model that is not covered by the Goldstein-Schlag paper and its successors; the Fibonacci Hamiltonian, introduced independently by Kohmoto et al.\ \cite{kkt} and Ostlund et al.\ \cite{oprss}. It is given by $\Omega = \T = \R/\Z$, $Tx = x + \alpha \mod 1$, where
$$
\alpha = \frac{\sqrt{5}-1}{2}
$$
is the inverse of the golden mean, $\mu$ is the Lebesgue measure on $\T$, and $f(\omega) = \lambda \chi_{[1-\alpha,1)}(\omega)$ for some $\lambda > 0$. Thus, the potentials have the form
\begin{equation}\label{fibpot}
V_\omega(n) = \lambda \chi_{[1-\alpha,1)}(n \alpha + \omega \mod
1).
\end{equation}
The associated operators $\{H_\omega\}_{\omega \in \Omega}$ form the family of Fibonacci Hamiltonians. This is the standard model of a one-dimensional quasicrystal. The spectrum of $H_\omega$ is easily seen to be independent of $\omega$, and it will henceforth be denoted by $\Sigma_\lambda$. This set is known to be a Cantor set of zero Lebesgue measure \cite{s5} (and in fact of Hausdorff dimension strictly between $0$ and $1$ \cite{Can}). Moreover, it is known that $H_\omega$ has purely singular continuous spectrum for every $\lambda$ and $\omega$; see Damanik and Lenz \cite{dl}, Kotani \cite{k}, and S\"ut\H{o} \cite{s3,s5}. The survey articles \cite{d, D07, D09, s3} contain information on the results obtained for this model and its generalizations. Here we are interested in the integrated density of states associates with the Fibonacci Hamiltonian, which we will henceforth denote by $N_\lambda$ since its dependence on the coupling constant $\lambda$ will be of explicit interest. We mention in passing that $dN_\lambda$ is the equilibrium measure on $\Sigma_\lambda$ in the sense of logarithmic potential theory.

\bigskip

We first note a result that is essentially well-known, but which is stated for the sake of completeness.

\begin{theorem}\label{global}
For every $\lambda > 0$, there are $C_\lambda < \infty$ and $\gamma_\lambda > 0$ such that
$$
| N_\lambda(E_1) - N_\lambda(E_2) | \le C_\lambda |E_1 - E_2|^{\gamma_\lambda}
$$
for every $E_1,E_2$ with $|E_1 - E_2| < 1$.
\end{theorem}

\begin{proof}
It follows from the definition \eqref{e.idsspectraldef}--\eqref{e.idsspectraldef2} that the integrated density of states is the distribution function of the $\mu$-average of the spectral measures with respect to $H_\omega$ and $\delta_0$. It was shown in \cite{DKL} that for each $\lambda$, these spectral measures are uniformly H\"older continuous with constants $C_\lambda < \infty$ and $\gamma_\lambda > 0$ that are uniform in $\omega$. That is, the $\mu$-average of these measures will also be uniformly H\"older continuous with the same pair of constants.
\end{proof}

One can infer explicit expressions for $C_\lambda$ and $\gamma_\lambda$ from \cite{DKL}. However, they are clearly far from optimal and hence we opted not to make them explicit.

\bigskip

Our main goal is to identify the asymptotic behavior of the H\"older exponent in the regimes of large and small coupling. As mentioned above, the tools used to establish Theorem~\ref{global} do not produce optimal results and hence are inadequate to identify the asymptotic behavior precisely. Therefore, different methods are needed in these asymptotic regimes, and we will indeed use different ones in either of these two cases.

\bigskip

In the large coupling regime, we have the following:

\begin{theorem}\label{main}
{\rm (a)} Suppose $\lambda > 4$. Then, for every
$$
\gamma < \frac{3\log(\alpha^{-1})}{2\log(2\lambda + 22)},
$$
there is some $\delta > 0$ such that the IDS associated with the family of Fibonacci Hamiltonians satisfies
$$
| N_\lambda(E_1) - N_\lambda(E_2) | \le |E_1 - E_2|^{\gamma}
$$
for every $E_1,E_2$ with $|E_1 - E_2| < \delta$.
\\[1mm]
{\rm (b)} Suppose $\lambda \ge 8$. Then, for every
$$
\tilde \gamma > \frac{3\log(\alpha^{-1})}{2\log \left(\frac{1}{2} \left( (\lambda - 4) + \sqrt{(\lambda - 4)^2 - 12} \right)\right)}
$$
and every $0 < \delta < 1$, there are $E_1,E_2$ with $0< |E_1 - E_2| < \delta$ such that
$$
| N_\lambda(E_1) - N_\lambda(E_2) | \ge |E_1 - E_2|^{\tilde \gamma}.
$$
\end{theorem}

This shows in particular that the optimal H\"older exponent is asymptotically $\frac{3\log(\alpha^{-1})}{2\log \lambda}$ in the large coupling regime.

Theorem~\ref{main} is proved in Subsection~\ref{ss.proofthmmain}. The proof is based on the self-similarity of the spectrum of $H_\omega$. In particular, we do not use the Thouless formula and a H\"older continuity result for the Lyapunov exponent, as was the case in many of the works mentioned above. Thus, in a sense, we use a geometric, rather than an analytic, approach. Before turning to the proof, we first recall the canonical periodic approximants, which are obtained by replacing $\alpha$ by its continued fraction approximants. This enables us to describe the self-similarity of the spectrum that is crucial to our proof and it will also establish an explicit way to express $N_\lambda(E)$ in terms of periodic spectra.

\bigskip

In the small coupling regime, we have the following:

\begin{theorem}\label{t.idsholdersmall}
The integrated density of states $N_{\lambda}(\cdot)$ is H\"older continuous with H\"older exponent $\gamma_{\lambda}$, where $\gamma_{\lambda}\to \frac{1}{2}$ as $\lambda\to 0$, and $\gamma_{\lambda}<\frac{1}{2}$ for small $\lambda>0$.

More precisely,

{\rm (a)} For any $\gamma \in (0, \frac{1}{2})$, there exists $\lambda_0 > 0$ such that for any $\lambda \in (0, \lambda_0)$, there exists $\delta>0$ such that
$$
|N_{\lambda}(E_1)-N_{\lambda}(E_2)|\le |E_1-E_2|^{\gamma}
$$
for every $E_1, E_2$ with $|E_1-E_2|<\delta$.

{\rm (b)} For any sufficiently small $\lambda > 0$, there exists $\tilde \gamma = \tilde \gamma (\lambda) < \frac{1}{2}$ such that for every $\delta > 0$, there are $E_1, E_2$ with $0<|E_1-E_2|<\delta$ and
$$
|N_{\lambda}(E_1)-N_{\lambda}(E_2)|\ge |E_1-E_2|^{\tilde\gamma}.
$$
\end{theorem}

\medskip

Theorem \ref{t.idsholdersmall} is proved in Subsection~\ref{s.three}. The proof uses the trace map formalism and the dynamical properties of the Fibonacci trace map studied previously in \cite{BR, Can, Cas, DG1, DG09b, DG2, Ro}. In particular, a relation between the integrated density of states of the Fibonacci Hamiltonian and the measure of maximal entropy of the trace map was established in \cite{DG3}. The proof combines this relation with H\"older structures that appear due to hyperbolicity of the trace map in order to get explicit estimates on the H\"older exponent. To show that the obtained asymptotics of the H\"older exponent are optimal, we study the behavior of unstable multipliers of specific periodic points of the trace map.

\bigskip

We conclude this introduction with some general remarks. In \cite{DG3} we studied the scaling exponents associated with the measures $dN_\lambda$. We showed that there exists $0 < \tilde \lambda_0 \le \infty$ such that for $\lambda \in (0,\tilde \lambda_0)$, there is $d_\lambda \in (0,1)$ so that the density of states measure $dN_\lambda$ is of exact dimension $d_\lambda$, that is, for $dN_\lambda$-almost every $E \in \R$, we have
$$
\lim_{\varepsilon \downarrow 0} \frac{\log N_\lambda(E - \varepsilon , E + \varepsilon)}{\log \varepsilon} = d_\lambda.
$$
Moreover,
$$
\lim_{\lambda \downarrow 0} d_\lambda = 1.
$$
While at first sight this result and the question addressed in Theorem~\ref{t.idsholdersmall} are quite similar, note that in the latter result, one has to establish a uniform estimate in the vicinity of an arbitrarily chosen energy in $\Sigma_\lambda$, whereas in the former result, one may exclude a set of zero $dN_\lambda$ measure from the consideration. This crucial difference leads to different answers, as the different asymptotics show. Of course the scaling exponent of $dN_\lambda$ at any energy $E \in \Sigma_\lambda$ bounds the optimal global H\"older exponent of $N_\lambda$ from above, and hence one gets a one-sided estimate in this way. More specifically, the scaling exponents are worse at gap boundaries of $\Sigma_\lambda$ and essentially determine the global H\"older exponent of $N_\lambda$, but these points form a countable set and hence a set of measure zero with respect to $dN_\lambda$ since this measure is always continuous.

It is interesting to compare the large coupling asymptotics of several $\lambda$-dependent quantities. As we saw in Theorem~\ref{main}, the optimal H\"older exponent behaves asymptotically like $1.5 \cdot \frac{\log(\alpha^{-1})}{\log \lambda}$. On the other hand, it was shown in \cite{DEGT} that the Hausdorff dimension of $\Sigma_\lambda$ behaves asymptotically like $1.831\ldots \cdot \frac{\log(\alpha^{-1})}{\log \lambda}$, and it was shown in \cite{DT07} that a certain transport exponent, which measures the rate of wavepacket spreading in the time-dependent Schr\"odinger equation associated with the Fibonacci Hamitonian, behaves asymptotically like $2 \cdot \frac{\log(\alpha^{-1})}{\log \lambda}$. Similarly, in the weak coupling regime, we have that the Hausdorff dimension of $\Sigma_\lambda$ strictly exceeds the dimension $d_\lambda$ of $dN_\lambda$ (a fact that was also proven in \cite{DG3}), which in turn strictly exceeds the optimal H\"older exponent of $N_\lambda$ (since they have different asymptotic values). This shows that the strongly coupled and the weakly coupled Fibonacci Hamiltonian serve as a good source of examples demonstrating that certain quantities associated with a discrete Schr\"odinger operator need not be identical. Moreover, the three different prefactors in the large coupling asymptotics ($1.5$, $1.831\ldots$, and $2$) correspond directly, and in a quite beautiful way, to the scaling properties exhibited by $\Sigma_\lambda$ (cf.~Subsection~\ref{ss.scaling}).

Finally, determining the correct H\"older exponent is one important ingredient in a recent study of the spacings of the zeros of a certain class of orthogonal polynomials (or, equivalently, the eigenvalues of $H_{\omega,[1,n]}$ in the notation introduced above) by Kr\"uger and Simon \cite{KS} and hence our results feed into their theory.

\medskip

\noindent\textit{Acknowledgments.} It is a pleasure to thank Mark Embree for help with symbolic calculations and Helge Kr\"uger and Barry Simon for useful comments. In particular we are grateful to Helge Kr\"uger for his question about the small coupling behavior of the H\"older exponent, which prompted us to prove Theorem~\ref{t.idsholdersmall}.

\section{The Large Coupling Regime}\label{s.large}

\subsection{Canonical Periodic Approximants and Scaling Properties}\label{ss.scaling}

In this subsection, we recall some known results for the Fibonacci Hamiltonian, its periodic approximants, and their spectra. The main tools we shall need in the sequel are summarized in Propositions~\ref{order}--\ref{measure2} below.

Define the sequence $(F_k)_{k \ge -1}$ of Fibonacci numbers by
$$
F_{-1} = 0, \; F_0 = 1, \; F_{k} = F_{k-1} + F_{k-2} \mbox{ for } k \ge 1.
$$
For $k \ge 1$, define
$$
x_k (E,\lambda) = \frac{1}{2}{\rm tr} \, M_k(E, \lambda),
$$
where
$$
M_k(E,\lambda) = \left( \begin{array}{cr} E - V_{\omega = 0}(F_k) &  -1 \\ 1 & 0 \end{array} \right) \times \cdots \times \left( \begin{array}{cr} E - V_{\omega = 0}(1) &  -1 \\ 1 & 0 \end{array} \right),
$$
that is, the transfer matrix for phase $\omega = 0$ and energy $E$ from the origin to the site $F_k$.

The matrices $M_k$ obey the recursion
\begin{equation}\label{matrec}
M_k(E,\lambda) = M_{k-2}(E,\lambda) M_{k-1}(E,\lambda),
\end{equation}
and as a consequence, % the quantities $x_{k}(E,\lambda) = \frac12 \mathrm{Tr} \, M_k(E,\lambda)$ obey the recursion
\begin{equation}\label{tracemap}
x_{k+1}(E,\lambda) = 2 x_{k}(E,\lambda) x_{k-1}(E,\lambda) - x_{k-2}(E,\lambda).
\end{equation}
With the definitions above, \eqref{matrec} and \eqref{tracemap} hold for $k \ge 3$. If we define
$$
M_{-1}(E,\lambda) = \left( \begin{array}{cr} 1 &  -\lambda \\ 0 & 1 \end{array} \right) \text{ and } M_{0}(E,\lambda) = \left( \begin{array}{cr} E &  -1 \\
1 & 0 \end{array} \right),
$$
these recursions extend to all $k \ge 1$.

The so-called trace map relation, \eqref{tracemap}, yields the invariant
\begin{equation}\label{invariant}
x_{k+1}(E,\lambda)^2 + x_k(E,\lambda)^2 + x_{k-1}(E,\lambda)^2 - 2 x_{k+1}(E,\lambda) x_k(E,\lambda) x_{k-1}(E,\lambda) = 1 +
\frac{\lambda^2}{4}.
\end{equation}
The identities \eqref{matrec}--\eqref{invariant} were proved by S\"ut\H{o} in \cite{s3}.

For fixed $\lambda$, define (leaving the dependence on $\lambda$ implicit)
$$
\sigma_k = \{ E \in \R : |x_k (E, \lambda)| \le 1 \}.
$$
The set $\sigma_k$ is actually equal to the spectrum of the Schr\"odinger operator $H$ whose potential $V_k$ results from $V_{\omega = 0}$ in \eqref{fibpot} by replacing $\alpha$ by $F_{k-1}/F_k$ (see \cite{s3}). Hence, $V_k$ is $F_k$-periodic, $\sigma_k \subset \R$, and it consists of $F_k$ bands (closed intervals).

Next, we recall some results of Damanik and Tcheremchantsev \cite{dt}, Killip, Kiselev, and Last \cite{kkl}, and Raymond \cite{r}. From now on, we shall always assume $\lambda > 4$ since we will make critical use of the fact that in this case, it follows from the invariant, \eqref{invariant}, that three consecutive half-traces cannot simultaneously be bounded in absolute value by $1$:
\begin{equation}\label{critical}
\forall \, \lambda > 4, \; \forall \, E,k \; : \; \max \{ |x_k(E,\lambda)|, |x_{k+1}(E,\lambda)|, |x_{k+2}(E,\lambda)| \} > 1.
\end{equation}

Following \cite{kkl}, we call a band $I_k \subset \sigma_k$ a type A band if $I_k \subset \sigma_{k-1}$ (and hence $I_k \cap (\sigma_{k+1} \cup \sigma_{k-2}) = \emptyset$). A band $I_k \subset \sigma_k$ is called a type B band if $I_k \subset \sigma_{k-2}$ (and therefore $I_k \cap \sigma_{k-1} = \emptyset$); compare Figure~1.

By definition of $M_{-1}$ and $M_0$, $\sigma_{-1} = \R$, $\sigma_0 = [-2,2]$, and $\sigma_1 = [\lambda - 2, \lambda + 2]$. Hence, $\sigma_0$ consists of a single band of type A, and $\sigma_1$ consists of a single band of type B.

\begin{figure}\label{Fig1}
\begin{center}
\setlength{\unitlength}{0.175in}
\begin{picture}(20,7)(0,0)
\put(0,0){\line(1,0){8}} \put(10,0){\hbox to 0mm{\hss$k-1$\hss}}
\put(12,0){\line(1,0){8}} \put(0,2){\line(1,0){8}}
\put(10,2){\hbox to 0mm{\hss$k$\hss}} \put(12,2){\line(1,0){8}}
\put(0,4){\line(1,0){8}} \put(10,4){\hbox to 0mm{\hss$k+1$\hss}}
\put(12,4){\line(1,0){8}} \put(0,6){\line(1,0){8}}
\put(10,6){\hbox to 0mm{\hss$k+2$\hss}} \put(12,6){\line(1,0){8}}
\linethickness{2pt} \put(1,0){\line(1,0){6}}
\put(2,2){\line(1,0){4}} \put(3,6){\line(1,0){2}}
\put(12.5,2){\line(1,0){7}} \put(14.5,4){\line(1,0){3}}
\put(13,6){\line(1,0){1}} \put(18,6){\line(1,0){1}}
\end{picture}
\end{center}
\caption{left: a type A band in $\sigma_k$; right: a type B band
in $\sigma_k$.}
\end{figure}
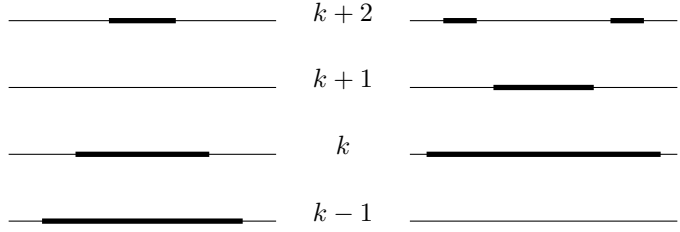

From \eqref{critical}, one gets the following result on the relative position of bands on successive levels (cf.~\cite[Lemma~5.3]{kkl} and \cite[Lemma~6.1]{r}):

\begin{prop}\label{order}
For every $\lambda > 4$ and every $k \ge 0$,
\begin{itemize}

\item[{\rm (a)}] Every type A band $I_k \subset \sigma_k$ contains exactly one type B band $I_{k+2} \subset \sigma_{k+2}$, and no other bands from $\sigma_{k+1}$, $\sigma_{k+2}$.

\item[{\rm (b)}] Every type B band $I_k \subset \sigma_k$ contains exactly one type A band $I_{k+1} \subset \sigma_{k+1}$ and two type B bands from
$\sigma_{k+2}$, positioned around $I_{k+1}$.

\end{itemize}
\end{prop}

The Lebesgue measure of each individual band admits the following geometric (in $k$) lower bound:

\begin{prop}\label{measure}
For $\lambda > 4$, $k \ge 3$, and every band $I_k$ of $\sigma_k$, we have
$$
|I_k| \ge \frac{4}{(2\lambda + 22)^{2k/3}},
$$
where $| \cdot |$ denotes Lebesgue measure.
\end{prop}

\begin{proof}
This follows from \cite[Lemma~3.5]{dt}, noting \eqref{critical} and the fact that on each band $I_k$ of $\sigma_k$, we have $\int_{I_k} |x'_k(E)| \, dE = 2$.\footnote{Here we caution the reader that the literature on the Fibonacci Hamiltonian is not consistent in the sense that some papers use half-traces, as we do here, and other papers use traces instead. The papers \cite{dt, kkl, r}, on which much of the present subsection is based, belong to the second group and hence their results need to be slightly reformulated when given here.}
\end{proof}

In the large coupling regime, this result is asymptotically optimal in the following sense:

\begin{prop}\label{measure2}
For $\lambda \ge 8$, $k \ge 4$ and $k \equiv 1 \mod 3$, there exists a type B band $I_k$ of $\sigma_k$ with
$$
|I_k| \le \frac{4}{\left(\frac{1}{2} \left( (\lambda - 4) + \sqrt{(\lambda - 4)^2 - 12} \right)\right)^{2k/3}}.
$$
\end{prop}

\begin{proof}
This follows from \cite[Lemma~5.5]{kkl}, Proposition~\ref{order} above (start with the type $B$ band of $\sigma_1$ and then cycle periodically through B, A, and ``not contained''), and again the fact that on each band $I_k$ of $\sigma_k$, we have $\int_{I_k} |x'_k(E)| \, dE = 2$.
\end{proof}

Let us study the scaling properties of the sets $\sigma_k$ in more detail. In particular, we want to show that, when $n > k$, each band of $\sigma_k$ contains either $F_{n-k}$ or $F_{n-k-2}$ bands of $\sigma_n$.

Recall that $\sigma_0$ consists of a single band of type A, and $\sigma_1$ consists of a single band of type B. A repeated application of Proposition~\ref{order} yields the entries in the following table:
$$
\begin{array}{|c|c|c|}

\hline

k & \# \text{ of bands in $\sigma_k$ of type A} & \# \text{ of bands in $\sigma_k$ of type B} \\

\hline

\hline

0 & 1 & 0 \\

\hline

1 & 0 & 1 \\

\hline

2 & 1 & 1 \\

\hline

3 & 1 & 2 \\

\hline

4 & 2 & 3 \\

\hline

5 & 3 & 5 \\

\hline

6 & 5 & 8 \\

\hline

\vdots & \vdots & \vdots \\

\hline

\end{array}
$$

We observe the following:

\begin{lemma}\label{l1}
For $\lambda > 4$ and $k \ge 2$, we have
\begin{align*}
\# \text{ of type A bands in $\sigma_k$} & = F_{k-2}, \\
\# \text{ of type B bands in $\sigma_k$} & = F_{k-1}.
\end{align*}
In particular, for $k \ge 0$, every band of $\sigma_k$ is either of type A or of type B.
\end{lemma}

\begin{proof}
This is a straightforward induction.
\end{proof}

A similar straightforward application of Proposition~\ref{order} gives the following result.

\begin{lemma}\label{l2}
Suppose $\lambda > 4$, $k \ge 0$, and $n > k$.\\
{\rm (a)} Every type A band of $\sigma_k$ contains $F_{n-k-2}$ bands of $\sigma_n$.\\
{\rm (b)} Every type B band of $\sigma_k$ contains $F_{n-k}$ bands of $\sigma_n$.
\end{lemma}

\subsection{Proof of Theorem~\ref{main}}\label{ss.proofthmmain}

In this subsection we prove Theorem~\ref{main}. Our main tools will be Proposition~\ref{measure}, Lemmas~\ref{l1} and \ref{l2}, and a formula, given in \eqref{idsbands} below, connecting the IDS of the Fibonacci Hamiltonian and the band structure of the periodic spectra introduced the previous subsection.

Let us first note that in the Fibonacci case, as a consequence of results of Hof for uniquely ergodic models, \cite{h}, the convergence in \eqref{idsform} takes place for every, rather than almost every, $\omega \in \Omega$. Moreover, Dirichlet boundary conditions can be replaced by other boundary conditions, such as Neumann or periodic boundary conditions. If we choose periodic boundary conditions, and consider $\omega = 0$ and convergence only along the subsequence $(F_n)_{n \ge 1}$, we obtain the following formula (which had already been noted by Raymond \cite{r}),
\begin{equation}\label{idsbands}
N_\lambda(E) = \lim_{n \to \infty} \frac{ \# \{ \text{bands of } \sigma_n
\le E \} }{F_n},
\end{equation}
since each band of $\sigma_n$ contains exactly one eigenvalue of the operator $H_{\omega = 0}$, restricted to $[1,F_n]$ with periodic boundary conditions.

\begin{proof}[Proof of Theorem~\ref{main}.]
(a) Given $E_1 < E_2$ (with $E_2 - E_1$ small; less than $4/(2\lambda + 22)^2$, say), we want to estimate $N_\lambda(E_2) - N_\lambda(E_1)$ from above. It follows from \eqref{idsbands} that
\begin{equation}\label{diff}
N_\lambda(E_2) - N_\lambda(E_1) =  \lim_{n \to \infty} \frac{ \# \{ \text{bands of $\sigma_n$ contained in } [E_1,E_2]  \} }{F_n}.
\end{equation}
Thus, we need to estimate from above the number of bands of $\sigma_n$ that are contained in $[E_1,E_2]$.

Let $k \ge 3$ be the integer with
$$
\frac{4}{(2\lambda + 22)^{2(k+1)/3}} \le E_2 - E_1 < \frac{4}{(2\lambda + 22)^{2k/3}}.
$$
Consider, for $n > k$, the bands of $\sigma_n$ that are contained in $[E_1,E_2]$. By the results of the previous section, each of these bands is contained in a band of $\sigma_k$ or in a band of $\sigma_{k-1}$. By the definition of $k$ and Proposition~\ref{measure}, at most two bands from $\sigma_k \cup \sigma_{k-1}$ can occur as associated bands. Thus, Lemmas~\ref{l1} and \ref{l2} imply that at most $2F_{n-(k-1)}$ bands of $\sigma_n$ can be contained in $[E_1,E_2]$.

Define $\gamma_k$ by
$$
\gamma_k = \frac{(k-3) \log (\alpha^{-1})}{\frac{2}{3}(k+1) \log( 2 \lambda + 22) - \log 4}.
$$
Then, we infer from \eqref{diff} that
\begin{align*}
N_\lambda(E_2) - N_\lambda(E_1) & =  \lim_{n \to \infty} \frac{ \# \{ \text{bands of $\sigma_n$ contained in } [E_1,E_2]  \} }{F_n}\\
& \le \lim_{n \to \infty} \frac{2 F_{n-(k-1)}}{F_n} \\
& \le \lim_{n \to \infty} \frac{F_{n-(k-3)}}{F_n} \\
& = \alpha^{k-3} \\
& = \left( \frac{4}{(2\lambda + 22)^{2(k+1)/3}} \right)^{\gamma_k} \\
& \le (E_2 - E_1)^{\gamma_k}.
\end{align*}

Now, given
$$
0 < \gamma < \frac{3\log(\alpha^{-1})}{2\log(2\lambda + 22)},
$$
choose $k_0 \ge 3$ such that $\gamma_k \ge \gamma$ for every $k \ge k_0$. Let
$$
\delta = \frac{4}{(2\lambda + 22)^{2k_0/3}} \in (0,1).
$$

We obtain that for every $E_1,E_2$ with $|E_1 - E_2| < \delta$,
$$
| N_\lambda(E_1) - N_\lambda(E_2) | \le  |E_1 - E_2|^{\gamma_k} \le |E_1 - E_2|^{\gamma},
$$
where $k \ge k_0$ is the integer associated with $|E_1 - E_2|$.

(b) Suppose $\lambda \ge 8$ and let
$$
\tilde \gamma \in \left( \frac{3\log(\alpha^{-1})}{2\log\left(\frac{1}{2} \left( (\lambda - 4) + \sqrt{(\lambda - 4)^2 - 12} \right)\right)} , 1 \right)
$$
and $0 < \delta < 1$ be given.

Choose $k_0$ such that
$$
\frac{4}{\left(\frac{1}{2} \left( (\lambda - 4) + \sqrt{(\lambda - 4)^2 - 12} \right)\right)^{2k_0/3}} < \delta
$$
and such that
$$
\tilde \gamma_k := \frac{(k+1) \log \alpha^{-1}}{\frac23 k \, \log \left( \frac{1}{2} \left( (\lambda - 4) + \sqrt{(\lambda - 4)^2 - 12} \right) \right) - \log 4} < \tilde \gamma
$$
for every $k \ge k_0$.

Next, choose $k \ge \max \{ k_0, 4 \}$ with $k \equiv 1 \mod 3$. By Proposition~\ref{measure2}, there exists a type B band $I_k$ of $\sigma_k$ with
$$
|I_k| \le \frac{4}{\left(\frac{1}{2} \left( (\lambda - 4) + \sqrt{(\lambda - 4)^2 - 12} \right)\right)^{2k/3}}.
$$
Denote the endpoints of $I_k$ by $E_1, E_2$, that is, $I_k = [E_1,E_2]$.

By \eqref{diff}, Lemma~\ref{l2}, the definition of $\tilde \gamma_k$, the choice of $I_k = [E_1,E_2]$, the choice of $k$, and the fact that $0 < E_2 - E_1 < \delta < 1$, we find that
\begin{align*}
N_\lambda(E_2) - N_\lambda(E_1) & =  \lim_{n \to \infty} \frac{ \# \{ \text{bands of $\sigma_n$ contained in } [E_1,E_2]  \} }{F_n}\\
& = \lim_{n \to \infty} \frac{F_{n-k-1}}{F_n} \\
& = \alpha^{k+1} \\
& = \left( \frac{4}{\left(\frac{1}{2} \left(
(\lambda - 4) + \sqrt{(\lambda - 4)^2 - 12} \right)\right)^{2k/3}} \right)^{\tilde \gamma_k} \\
& \ge (E_2 - E_1)^{\tilde \gamma_k} \\
& \ge (E_2 - E_1)^{\tilde \gamma}.
\end{align*}
This completes the proof.
\end{proof}

\section{The Small Coupling Regime}

Here we will use the relation between the IDS for the Fibonacci Hamiltonian and the measure of maximal entropy for the so called Trace Map associated with the discrete Schr\"odinger operator with Fibonacci potential.

\subsection{IDS for the Free Laplacian}

It is well known that $\Sigma_0 = [-2,2]$ and
\begin{equation}\label{e.freeids}
N_0(E) = \begin{cases} 0 & E \le -2 \\ \frac{1}{\pi} \arccos \left( - \frac{E}{2} \right) & -2 < E < 2 \\ 1 & E \ge 2. \end{cases}
\end{equation}

In particular, $N_0$ is H\"older continuous with H\"older exponent $1/2$. We will need a somewhat more detailed description of the continuity properties of $N_0$ (which also follows directly from the explicit form \eqref{e.freeids}).

\blm\label{l.IDSLap}
For any $\varepsilon>0$, there exists $C>0$ such that the following hold.

{\rm (a)} If $E_1, E_2\in [-2+\varepsilon, 2-\varepsilon]$, then
$$
|N_0(E_2)-N_0(E_1)|\le C|E_2-E_1|.
$$

{\rm (b)} If $E_1, E_2\in [-2, 2]$, then
$$
|N_0(E_2)-N_0(E_1)|\le C|E_2-E_1|^{1/2}.
$$

{\rm (c)} If $E_1, E_2\in [-2, -2+\varepsilon)$ and $-2<E_1<E_2$, then
$$
|N_0(E_2)-N_0(E_1)|\le \frac{C}{|2+E_1|^{1/2}}|E_2-E_1|.
$$

Similarly, if $E_1, E_2\in (2-\varepsilon, 2]$ and $E_1<E_2<2$, then
$$
|N_0(E_2)-N_0(E_1)|\le \frac{C}{|2-E_2|^{1/2}}|E_2-E_1|.
$$

{\rm (d)} There exists $C_0>0$ such that for any $E\in (-2, 2)$, one has
$$
|N_0(-2)-N_0(E)|\ge C_0|E+2|^{1/2} \ \ \ \ \text{and}\ \ \ \ |N_0(2)-N_0(E)|\ge C_0|E-2|^{1/2}.
$$
\elm

\subsection{The Trace Map}

The recursion \eqref{tracemap} gives rise to a fundamental connection between the spectral properties of the Fibonacci Hamiltonian and the dynamics of the \textit{trace map}
$$
T : \Bbb{R}^3 \to \Bbb{R}^3, \; T(x,y,z)=(2xy-z,x,y).
$$
The function $G(x,y,z) = x^2+y^2+z^2-2xyz-1$ is invariant\footnote{It is usually called the Fricke-Vogt invariant.} under the action of $T$ (cf.~\eqref{invariant}), and hence $T$ preserves the family of cubic surfaces\footnote{The surface $S_0$ is known as Cayley cubic.}
$$
S_\lambda = \left\{(x,y,z)\in \Bbb{R}^3 : x^2+y^2+z^2-2xyz=1+ \frac{\lambda^2}{4} \right\}.
$$
It is therefore natural to consider the restriction $T_{\lambda}$ of the trace map $T$ to the invariant surface $S_\lambda$. That is, $T_{\lambda}:S_\lambda \to S_\lambda$, $T_{\lambda}=T|_{S_\lambda}$. We denote by $\Lambda_{\lambda}$ the set of points in $S_\lambda$ whose full orbits under $T_{\lambda}$ are bounded (it is known that $\Lambda_\lambda$ is equal to the non-wandering set of $T_\lambda$).

\subsection{Hyperbolicity of the Trace Map}

Recall that an invariant closed set $\Lambda$ of a diffeomorphism $f : M \to M$ is \textit{hyperbolic} if there exists a splitting of the tangent space $T_xM=E^u_x\oplus E^u_x$ at every point $x\in \Lambda$ such that this splitting is invariant under $Df$, the differential $Df$ exponentially contracts vectors from the stable subspaces $\{E^s_x\}$, and the differential of the inverse, $Df^{-1}$, exponentially contracts vectors from the unstable subspaces $\{E^u_x\}$. A hyperbolic set $\Lambda$ of a diffeomorphism $f : M \to M$ is \textit{locally maximal} if there
exists a neighborhood $U$ of $\Lambda$ such that
$$
\Lambda=\bigcap_{n\in\Bbb{Z}}f^n(U).
$$
It is known that for $\lambda > 0$, $\Lambda_{\lambda}$ is a locally maximal hyperbolic set of $T_{\lambda} : S_\lambda \to S_\lambda$; see \cite{Can, Cas, DG1}.

\subsection{Properties of the Trace Map for $\lambda=0$}\label{ss.vequzero}

The surface
$$
\mathbb{S} = S_0 \cap \{ (x,y,z)\in \Bbb{R}^3 : |x|\le 1, |y|\le 1, |z|\le 1\}
$$
is homeomorphic to $S^2$, invariant under $T$, smooth everywhere except at the four points $P_1=(1,1,1)$, $P_2=(-1,-1,1)$, $P_3=(1,-1,-1)$, and $P_4=(-1,1,-1)$, where $\mathbb{S}$ has conic singularities, and the trace map $T$ restricted to $\mathbb{S}$ is a factor of the hyperbolic automorphism of $\T^2 = \R^2 / \Z^2$ given by
$$
\mathcal{A}(\theta, \varphi) = (\theta + \varphi, \theta)\ (\text{\rm mod}\ 1).
$$
The semi-conjugacy is given by the map
\begin{equation}\label{e.semiconj}
F: (\theta, \varphi) \mapsto (\cos 2\pi(\theta + \varphi), \cos 2\pi \theta, \cos 2\pi \varphi).
\end{equation}
The map $\mathcal{A}$ is hyperbolic, and is given by the matrix $A = \begin{pmatrix} 1 & 1 \\ 1 & 0 \end{pmatrix}$, which has eigenvalues
$$
\mu=\frac{1+\sqrt{5}}{2}\ \ \text{\rm and} \ \ \ -\mu^{-1}=\frac{1-\sqrt{5}}{2}.
$$

\begin{figure}[htb]
{\begin{minipage}{5cm}
\includegraphics[width=1.2\textwidth]{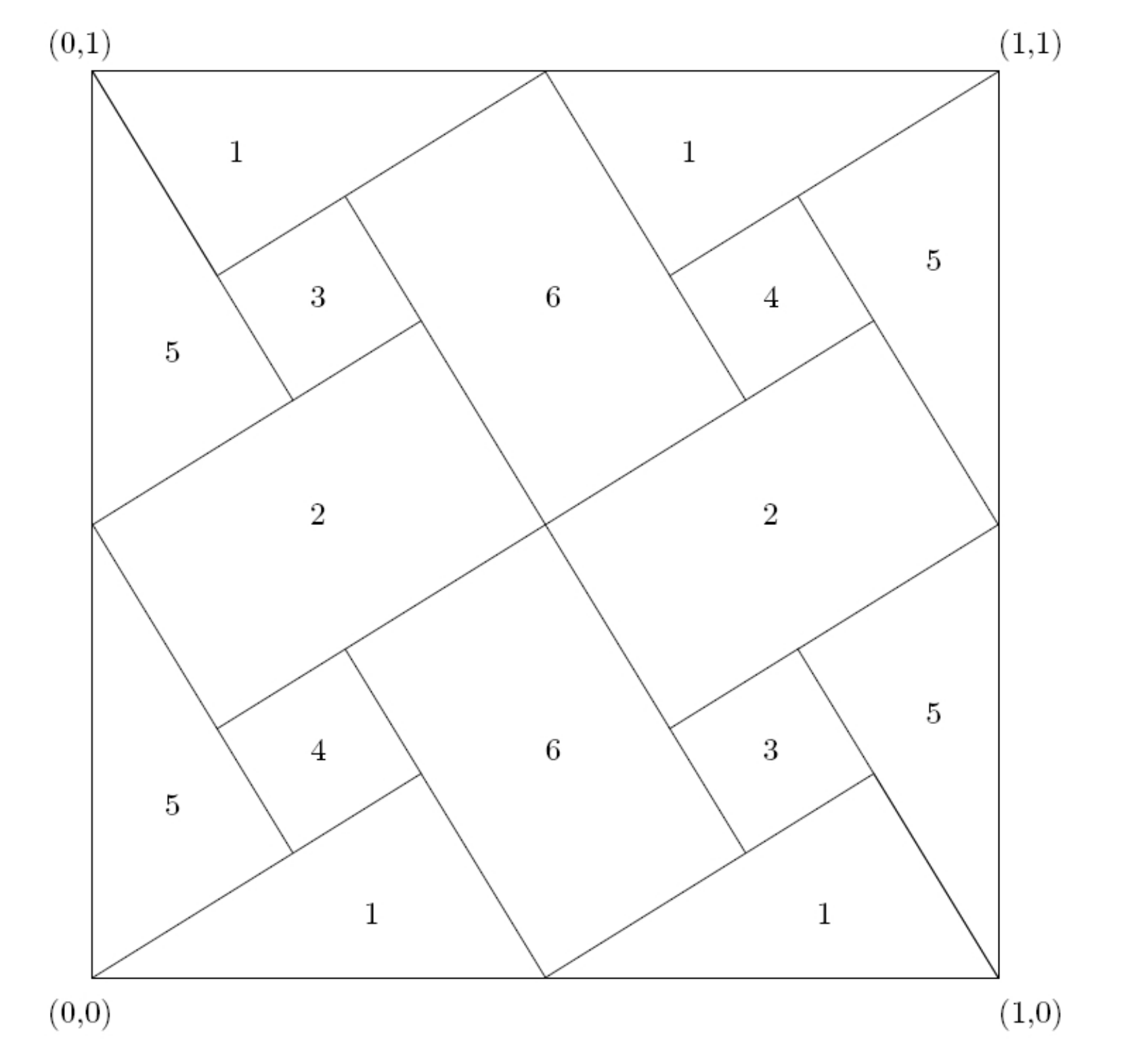}
\label{fig:sing0.0} \end{minipage} \; \; \; \; \; {\Large $\xrightarrow{F}$} \ \
\begin{minipage}{5cm}
\includegraphics[width=1.25\textwidth]{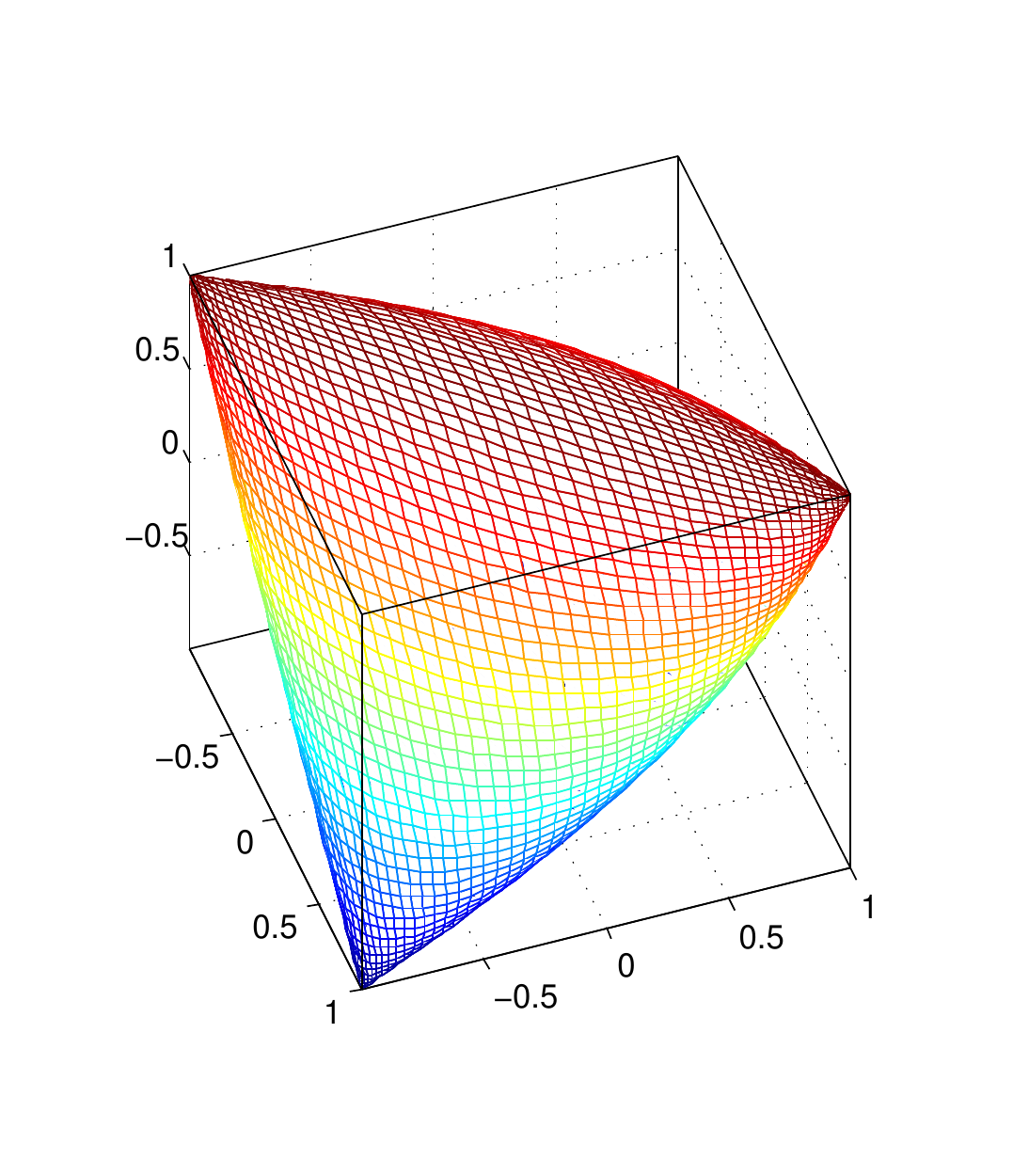}
\end{minipage}}
\caption{The semi-conjugacy $F$ between the linear map $\mathcal{A}$ and the trace map $T$ on the central part $\mathbb{S}$ of the Cayley cubic.}\label{fig.semiconj}
\end{figure}

A Markov partition for the map $\mathcal{A}:\mathbb{T}^2\to \mathbb{T}^2$ is shown in Figure~\ref{fig.semiconj}. Its image under the map $F:\mathbb{T}^2\to \mathbb{S}$ is a Markov partition for the pseudo-Anosov map $T:\mathbb{S}\to \mathbb{S}$.

\subsection{Spectrum and Trace Map}

Denote by $\ell_\lambda$ the line
$$
\ell_\lambda = \left\{ \left(\frac{E-\lambda}{2}, \frac{E}{2}, 1 \right) : E \in \Bbb{R} \right\}.
$$
It is easy to check that $\ell_\lambda \subset S_\lambda$. An energy $E \in \Bbb{R}$ belongs to the spectrum $\Sigma_\lambda$ of the Fibonacci Hamiltonian if and only if the positive semiorbit of the point $(\frac{E-\lambda}{2}, \frac{E}{2}, 1)$ under iterates of the trace map $T$ is bounded; see \cite{s3}. Moreover, the stable manifolds of points in $\Lambda_\lambda$ intersect the line $\ell_\lambda$ transversally if $\lambda > 0$ is sufficiently small \cite{DG1} or if $\lambda \ge 16$ \cite{Cas}.

Let us denote $L_{\lambda}:\mathbb{R} \to \ell_{\lambda}, L_{\lambda}(E)=\left(\frac{E-\lambda}{2}, \frac{E}{2}, 1\right)$. It is affine and contracts distances by the multiplicative factor $\frac{1}{\sqrt{2}}$.

Define the map $\Psi_{\lambda}:\Sigma_{\lambda}\to [0,2]$  by
\beq\label{e.psi}
\Psi_{\lambda}(x)=y \ \ \ \Leftrightarrow \ \ \ N_{\lambda}(x)=N_{0}(y).
\eneq

It turns out that there is a direct relation between the map $\Psi_\lambda$ and dynamical structures of the trace map $T_\lambda$. The following statement is implicitly contained in \cite[Claim 3.2]{DG3}.

\bprop\label{p.dynpsi}
There exists $\lambda_0>0$ such that the following holds. Take any $\lambda \in [0, \lambda_0)$ and $x_1, x_2\in \Omega_\lambda$. Consider the stable manifolds $W^s(x_1)$ and $W^s(x_2)$, and take some points $p_1=W^s(x_1)\cap \ell_\lambda$ and $p_2=W^s(x_2)\cap \ell_\lambda$. When $\lambda$ changes, there are unique continuations of the points $x_1, x_2 \in \Omega_\lambda$, denote them by $x_1(\lambda), x_2(\lambda)$. The continuations of the intersections $p_i(\lambda) = W^s(x_i(\lambda))\cap \ell_\lambda$, $i=1,2$, are also well defined, and the value of the difference
$
N_\lambda(L_{\lambda}^{-1}(p_2(\lambda))-N_\lambda(L_{\lambda}^{-1}(p_1(\lambda))
$
is independent of $\lambda\in [0, \lambda_0)$.
\enprop

Notice that Proposition \ref{p.dynpsi} gives a dynamical description of the map $\Psi_\lambda$. In \cite{DG3} this description was used to establish a relation between the IDS $N_\lambda$ and the measure of maximal entropy for $T|_{\Lambda_\lambda}$.

\subsection{H\"older Continuity of the IDS in the Small Coupling Regime}\label{s.three}

Here we prove Theorem \ref{t.idsholdersmall}.

\begin{proof}[Proof of Theorem \ref{t.idsholdersmall}.]
Choose sufficiently small neighborhoods $U(P_i)$ of the singularities $\{P_1, P_2, P_3, P_4\}$ of the Cayley cubic. Let $U^*(P_i)\subset U(P_i)$ be an essentially smaller neighborhood of the singularity $P_i$. Set $U = \bigcup_{i=1, 2, 3, 4}U(P_i)$ and $U^* = \bigcup_{i=1, 2, 3, 4}U^*(P_i)$. Take any $E_1, E_2 \in \Sigma_{\lambda}$, and denote by $b$ the interval on $\ell_\lambda$ between the points $L_{\lambda}(E_1)$ and $L_{\lambda}(E_2)$.

Notice that
$$
|N_{\lambda}(E_1)-N_{\lambda}(E_2)|=|N_0(\Psi_{\lambda}(E_1))-N_0(\Psi_{\lambda}(E_2))|.
$$
Let $a$ be the interval on $\ell_0$ between the points $L_{0}(\Psi_{\lambda}(E_1))$ and $L_{0}(\Psi_{\lambda}(E_2))$. We will consider separately three cases:
\begin{itemize}

\item[(i)] $a$ and $b$ are away from the neighborhoods $U^*(P_i)$ of the singularities.

\item[(ii)] $a$ and $b$ are in a neighborhood $U(P_i)$ of a singularity (that could be either $P_1=(1, 1, 1)$ or $P_2=(-1, -1, 1)$), and one of the edges of $a$ is on a local strong stable manifold of the singularity (which implies that one of the edges of $b$ is on the local strong stable manifold of a periodic orbit of period 2 or 6). This is equivalent to the case when $E_1$ or $E_2$ is equal to $\min \Sigma_\lambda$ or $\max \Sigma_\lambda$.

\item[(iii)] $a$ and $b$ are in a neighborhood $U(P_i)$ of a singularity (that could be either $P_1=(1, 1, 1)$ or $P_2=(-1, -1, 1)$), and none of the edges of $a$ is on a local strong stable manifold of the singularity.

\end{itemize}

Note that if $|E_1-E_2|$ is sufficiently small (depending on the choice of $U$ and $U^*$), then exactly one of the cases (i)--(iii) applies.

Consider the case (i). Let us iterate $a$ and $b$ until they grow up to the length of order one. To simplify the estimates, let us introduce the following notation:
$$
x\sim_C y\ \ \ \ \Leftrightarrow\ \ \ \ C^{-1}|y|\le |x|\le C|y|.
$$
There exists a constant $C>0$ that is independent of $a, b$ such that for some $M\in \mathbb{N}$, we have $|T^M(a)|\sim_C|T^M(b)|\sim_C 1$. Let us split the sequence of iterates $\{a, T(a), T^2(a), \ldots, T^M(a)\}$ into finite intervals $\{a, T(a), \ldots, T^{k_1}(a)\}$, $\{T^{k_1+1}(a), \ldots, T^{k_2}(a)\}, \ldots$ (and similarly for $\{b, T(b), T^2(b), \ldots, T^M(b)\}$) in such a way that for each $k_i$, one of the following cases holds:

 (a) $\{T^{k_{i-1}+1}(a), \ldots, T^{k_i}(a)\}$ as well as $\{T^{k_{i-1}+1}(b), \ldots, T^{k_i}(b)\}$ are away from $U^*$, and moreover,  $\frac{|T^{k_i}(a)|}{|T^{k_{i-1}}(a)|}>2$ and $\frac{|T^{k_i}(b)|}{|T^{k_{i-1}}(b)|}>2$,

 (b) $\{T^{k_{i-1}+1}(a), \ldots, T^{k_i}(a)\}$ as well as $\{T^{k_{i-1}+1}(b), \ldots, T^{k_i}(b)\}$ are inside of $U$.

  In the  case (a), since for small $\lambda>0$, the maps $T_0|_{\mathbb{S}\backslash U^*}$ and $T_{\lambda}|_{{S_{\lambda}}\backslash U^*}$ are $C^1$-close, for some $\alpha_{\lambda} < 1$ with $\alpha_{\lambda}\to 1$ as $\lambda\to 0$, we have $\frac{|T^{k_i}(a)|}{|T^{k_{i-1}}(a)|}\ge \left(\frac{|T^{k_i}(b)|}{|T^{k_{i-1}}(b)|}\right)^{\alpha_{\lambda}}$. In the case (b), by \cite[Proposition 3.15]{DG2}, we have $\frac{|T^{k_i}(a)|}{|T^{k_{i-1}}(a)|}\ge C\mu_0^{\frac{k_i-k_{i-1}}{2}}$, where $\mu_0$ is an unstable multiplier at a singularity, and at the same time we have $\frac{|T^{k_i}(b)|}{|T^{k_{i-1}}(b)|}\le \tilde C\mu_{\lambda}^{k_i-k_{i-1}}$, where $\mu_{\lambda}$ is an unstable multiplier of a periodic orbit $P_{\lambda}$, and which obeys $\mu_{\lambda}\to \mu_0$ as $\lambda\to 0$. Choose $\gamma_{\lambda}<\frac{1}{2}$, $\gamma_\lambda\to \frac{1}{2}$ as $\lambda\to 0$, in such a way that $\mu_{\lambda}^{2\gamma_\lambda}<\mu_0$. Then, (since in this case, the smallness of $U^*$ relative to $U$ guarantees that $k_i-k_{i-1}$ is large enough) we have  $\frac{|T^{k_i}(a)|}{|T^{k_{i-1}}(a)|}\ge \left(\frac{|T^{k_i}(b)|}{|T^{k_{i-1}}(b)|}\right)^{\gamma_{\lambda}}$.

Now we have
$$
1\sim_C |T^M(a)|=|a|\cdot \frac{|T^{k_1}(a)|}{|a|}\cdot \frac{|T^{k_2}(a)|}{|T^{k_1}(a)|}\cdot \ldots \cdot \frac{|T^{k_m}(a)|}{|T^{k_{m-1}}(a)|}
$$
and
$$
1\sim_C |T^M(b)|=|b|\cdot \frac{|T^{k_1}(b)|}{|b|}\cdot \frac{|T^{k_2}(b)|}{|T^{k_1}(b)|}\cdot \ldots \cdot \frac{|T^{k_m}(b)|}{|T^{k_{m-1}}(b)|}.
$$
Therefore,
\begin{multline}\label{e.ab}
 |a|=\frac{|T^M(a)|}{\prod_{j=1}^{m}\frac{|T^{k_j}(a)|}{|T^{k_{j-1}}(a)|}}\le \\ \le \frac{C}{\left(\prod_{j=1}^{m}\frac{|T^{k_j}(b)|}{|T^{k_{j-1}}(b)|}\right)^{\gamma_{\lambda}}}=\frac{C}{\left(\frac{|T^M(b)|}{|b|}\right)^{\gamma_{\lambda}}}\le C^{1+\gamma_{\lambda}}|b|^{\gamma_{\lambda}}=C'|b|^{\gamma_{\lambda}}.
\end{multline}
%\beq\label{e.ab}
%|a|=\frac{|T^M(a)|}{\prod_{j=1}^{m}\frac{|T^{k_j}(a)|}{|T^{k_{j-1}}(a)|}}\le \frac{C}{\left(\prod_{j=1}^{m}\frac{|T^{k_j}(b)|}{|T^{k_{j-1}}(b)|}\right)^{\gamma_{\lambda}}}=\frac{C}{\left(\frac{|T^M(b)|}{|b|}\right)^{\gamma_{\lambda}}}\le C^{1+\gamma_{\lambda}}|b|^{\gamma_{\lambda}}=C'|b|^{\gamma_{\lambda}}.
%\eneq
Since $|b|=\frac{1}{\sqrt{2}}|E_1-E_2|$ and $|a|=\frac{1}{\sqrt{2}}|\Psi_{\lambda}(E_1)-\Psi_{\lambda}(E_2)|$, using Lemma~\ref{l.IDSLap}.(a), we find
\begin{multline*}
|N_\lambda(E_1)-N_\lambda(E_2)|=|N_0(\Psi_{\lambda}(E_1))-N_0(\Psi_{\lambda}(E_2))|\le \\ \le C|\Psi_{\lambda}(E_1)-\Psi_{\lambda}(E_2)|\le C\sqrt{2}|a|\le CC'\sqrt{2}|b|^{\gamma_{\lambda}}\le C''|E_1-E_2|^{\gamma_{\lambda}}.
\end{multline*}

Now let us consider the case (ii). In this case, one of the edges of $a$ and $b$ will never leave the neighborhood $U^*$. Let us iterate $a$ and $b$ sufficiently many times to have
\beq\label{e.M}
|T^M(a)|\sim_C |T^M(b)|\sim_C 1.
\eneq
Consider the coordinate system in a neighborhood of the singularity that rectifies all the invariant manifolds (see \cite[Section 4]{DG1} or \cite[Subsection 3.2]{DG2}). The lines $\ell_0$ and $\ell_{\lambda}$ are transversal to the central-stable manifolds, hence we can apply \cite[Proposition 3.15]{DG2} with some bounded (independent of $a$ and $b$) number $k^*$ from \cite[Lemma 3.16]{DG2}. Since $\mu_{\lambda}^{2\gamma_\lambda}<\mu_0$, we have
$$
|a|\le C\mu_0^{-M}\le C\mu_{\lambda}^{-2\gamma_{\lambda}M}\le C'|b|^{2\gamma_{\lambda}}.
$$
Using Lemma \ref{l.IDSLap}.(b), we find
\begin{multline*}
|N_\lambda(E_1)-N_\lambda(E_2)|=|N_0(\Psi_{\lambda}(E_1))-N_0(\Psi_{\lambda}(E_2))|\le \\ \le C|\Psi_{\lambda}(E_1)-\Psi_{\lambda}(E_2)|^{1/2}\le  C'|E_1-E_2|^{\gamma_{\lambda}}.
\end{multline*}

Finally, let us consider the case (iii). Suppose that $E_1, E_2\in \Sigma_{\lambda}$ are such that $E_1 < E_2 < \max \Sigma_{\lambda}$ and $\Psi_{\lambda}(E_1)$ is close to 2 (the case when $\Psi_{\lambda}(E_1)$ is close to $-2$ is similar). Then, $a\subset \ell_0$, $b\subset \ell_\lambda$ are in the neighborhood $U(P_1)$ of $P_1=(1, 1, 1)$, and the distance from $a$ to the central-stable manifold is of order $\mu_0^{-s}$, where $s$ is a number of iterates needed for $a$ to leave $U(P_1)$; see \cite[Proposition 3.14]{DG2}. This implies that
$$
|2-\Psi_\lambda(E_2)|\sim_C\mu_0^{-s}.
$$
Consider the arcs $\tilde a=T^s(a)$ and $\tilde b=T^s(b)$. These arcs are away from $U$, and therefore from (\ref{e.ab}) we deduce that $|\tilde a|\le C|\tilde b|^{\gamma_{\lambda}}$.
On the other hand, we have
$$
|\tilde a|\sim_C\mu_0^s|a|, \ \ \ |\tilde b|\sim_C\mu_\lambda^s|b|,
$$
and therefore %\marginpar{several changes here}
\begin{multline*}
    |a|\le C\mu_0^{-s}|\tilde a|\le C^2\mu_0^{-s}|\tilde b|^{\gamma_{\lambda}}\le C^{2+\gamma_{\lambda}}\mu_0^{-s}\mu_{\lambda}^{s\gamma_{\lambda}}|b|^{\gamma_\lambda} = C^{2+\gamma_{\lambda}}\mu_0^{-\frac{s}{2}}\left(\frac{\mu_{\lambda}^{\gamma_\lambda}}{\mu_0^{1/2}}\right)^s|b|^{\gamma_\lambda}\le \\
    \le C'\mu_0^{-\frac{s}{2}}|b|^{\gamma_\lambda}.
\end{multline*}

Now, using Lemma \ref{l.IDSLap}.(c), we find
\begin{multline*}
|N_\lambda(E_1)-N_\lambda(E_2)|=|N_0(\Psi_{\lambda}(E_1))-N_0(\Psi_{\lambda}(E_2))|\le \\ \le \frac{C}{|2-\Psi_\lambda(E_2)|^{\frac{1}{2}}}|\Psi_{\lambda}(E_1)-\Psi_{\lambda}(E_2)|\le \frac{C^2C'}{\mu_0^{-\frac{s}{2}}}\mu_0^{-\frac{s}{2}}|b|^{\gamma_\lambda}\le C''|E_1-E_2|^{\gamma_{\lambda}}.
\end{multline*}
Finally, notice that in all cases (i), (ii), (iii), the constant in the inequality can be set to $1$ if one takes a slightly smaller H\"older exponent and sufficiently small $|E_1-E_2|$. This finishes the proof of part (a) of Theorem \ref{t.idsholdersmall}.

In order to show part (b) of Theorem \ref{t.idsholdersmall}, consider the periodic points of period $2$ that are born out of the singularity $(1, 1, 1)$. The strong stable manifolds of these points correspond to the boundaries of the gaps in the spectrum. These periodic points form a curve
$$
Per_2=\left\{(x, y, z) : x \in \left(-\infty, \frac{1}{2}\right)\cup \left(\frac{1}{2}, \infty\right), y=\frac{x}{2x-1}, z=x, \right\}.
$$
For the map $T^2$, these points are fixed points, and $DT^2(1, 1, 1)=\begin{pmatrix}
                                                                      6 & 3 & -2 \\
                                                                      2 & 2 & -1 \\
                                                                      1 & 0 & 0 \\
                                                                    \end{pmatrix}$, with eigenvalues $1$ and $\frac{7\pm 3\sqrt{5}}{2}$.

\blm\label{l.muin}
For $\lambda > 0$ sufficiently small, the largest eigenvalue of $DT^2$ at the periodic point of period 2 near the singularity $(1, 1, 1)$ is strictly larger than the largest eigenvalue of $DT^2(1, 1, 1)$.
\elm

\begin{proof}
Take a periodic point $\left(x, \frac{x}{2x-1}, x \right)\in Per_2$.  We have
$$
DT^2\left(x, \frac{x}{2x-1}, x \right)=\begin{pmatrix}
                                         \frac{2x(4x-1)}{2x-1} & \frac{4x-1}{(2x-1)^2} & \frac{2x}{1-2x} \\
                                         2x & \frac{2x}{2x-1} & -1 \\
                                         1 & 0 & 0 \\
                                       \end{pmatrix},
$$
and the largest eigenvalue is
$$
\lambda^u(x)=\frac{1-2x+8x^2+\sqrt{-3+12x+4x^2-32x^3+64x^4}}{2(2x-1)}.
$$
Now we have
\begin{multline*}
    \frac{d}{dx}\lambda^u(x)|_{x=1}=0, \ \frac{d^2}{dx^2}\lambda^u(x)|_{x=1}=
    \\
    =\frac{8(-3(1+3\sqrt{5})+2(9+18\sqrt{5}+2(-27+3\sqrt{5}+2(55-12\sqrt{5}+4(-1+6\sqrt{5})))))}{135\sqrt{5}}=\\ =16.247987...>0
\end{multline*}
\end{proof}

In other words, Lemma \ref{l.muin} claims that $\mu_\lambda>\mu_0$ if $\lambda>0$ is small enough. Fix a small $\lambda>0$ and take any $\tilde\gamma\in (0, \frac{1}{2})$ such that $\mu_{\lambda}^{2\tilde\gamma}>\mu_0$. We claim that if $E_2 = \max \Sigma_\lambda$, $E_1 \in \Sigma_\lambda$, and $|E_1-E_2|$ is sufficiently small, then
$$
|N_{\lambda}(E_1)-N_{\lambda}(E_2)|\ge |E_1-E_2|^{\tilde\gamma}.
$$
Indeed, $\Psi(E_2)=2$, and the interval $a\subset \ell_0$ between the points $L_0(\Psi_{\lambda}(E_2))$ and $L_0(\Psi_{\lambda}(E_1))$ has one of its end points at the singularity $P_1=(1,1,1)$. Consider also the interval $b\subset \ell_\lambda$ between $L_\lambda(E_2)$ and $L_\lambda(E_1)$. As in the case (ii) above, consider $M$ iterates of $a$ and $b$, where $M$ is such that (\ref{e.M}) holds. Then, due to \cite[Proposition 3.14]{DG2}, we have
$$
|a|\sim_C \mu_0^{-M}, \ \ \ |b|\sim_C \mu_{\lambda}^{-M}.
$$
Using Lemma \ref{l.IDSLap}.(d), we find
\begin{multline*}
    |N_{\lambda}(E_1)-N_{\lambda}(E_2)|=|N_0(\Psi_{\lambda}(E_1))-N_0(\Psi_{\lambda}(E_2))|\ge \\ \ge  C_0|\Psi_\lambda(E_1)-\Psi_\lambda(E_2)|^{1/2} =2^{1/4}C_0|a|^{1/2}\ge C'\mu_0^{-\frac{M}{2}}=C'\left(\frac{\mu_{\lambda}^{M\tilde\gamma}}{\mu_0^{\frac{M}{2}}}\right)\mu_{\lambda}^{-M\tilde\gamma}\ge \\ \ge C''\left(\frac{\mu_{\lambda}^{2\tilde\gamma}}{\mu_0}\right)^{\frac{M}{2}}|b|^{\tilde\gamma}\ge C'''\left(\frac{\mu_{\lambda}^{2\tilde\gamma}}{\mu_0}\right)^{\frac{M}{2}}|E_2-E_1|^{\tilde\gamma}\ge |E_2-E_1|^{\tilde\gamma},
\end{multline*}
provided $M$ is large enough (i.e., if $|a|$ and $|b|$ are small enough, or, equivalently, if $|E_2-E_1|$ is small enough). This completes the proof of Theorem \ref{t.idsholdersmall}.
\end{proof}

Notice that as a byproduct of this proof, we also get the following statement:

\bprop\label{p.holderpsi}
The map $\Psi_{\lambda}:\Sigma_{\lambda}\to [0,2]$ given by \eqref{e.psi}
%$$
%\Psi_{\lambda}(x)=y \ \ \ \Leftrightarrow \ \ \ N_{\lambda}(x)=N_{0}(y)
%$$
is H\"older continuous with a H\"older exponent $\gamma_{\lambda}$ that obeys $\gamma_{\lambda}\to \frac{1}{2}$ as $\lambda\to 0$.
\enprop

\brm
In terms of the dynamics of the trace map, there exists $\lambda_0 > 0$ such that for any $\lambda\in (0, \lambda_0)$, the semiconjugacy $\Phi_{\lambda} : \Lambda_{\lambda} \to \mathbb{S}$, $\Phi_{\lambda}\circ T_{\lambda}|_{\Lambda_{\lambda}}=T_0|_{\mathbb{S}}\circ \Phi_{\lambda}$, is H\"older continuous with a H\"older exponent $\gamma_{\lambda}$ such that $\gamma_{\lambda}\to \frac{1}{2}$ as $\lambda\to 0$. Since we are not using this statement here, we do not elaborate on it.

This is related to the following classical fact from hyperbolic dynamics. Suppose $\Lambda_f$ is a compact locally maximal hyperbolic set of a surface diffeomorphism $f:M^2\to M^2$, $g$ is $C^1$-close to $f$, and $\Lambda_g$ is a continuation of $\Lambda_f$. Then there is a continuous conjugacy $h:\Lambda_f\to \Lambda_g$, $h\circ f=g\circ h$, and $h$ has to be H\"older continuous with H\"older exponent close to one {\rm (}see \cite{KaH}, \cite{PV}{\rm )}. In our case, the H\"older exponent of the conjugacy $\Phi_\lambda$ is close to $1/2$, not to one, due to essentially different behavior of $T_0|_{\mathbb{S}}$ and $T_{\lambda}|_{\Lambda_\lambda}$ near the singularities of the Cayley cubic.
\erm

\end{document}